\documentclass[a4paper,twoside,reqno]{amsart}
\usepackage{euscript,hyperref,graphics,color,slashed}
\usepackage{graphicx}
\usepackage{mathrsfs}
\usepackage{comment}
\usepackage{import}
\usepackage{tikz}
\usepackage{latexsym}
\usepackage[makeroom]{cancel}
\usepackage{amsthm, amssymb, amsmath, amsfonts, mathrsfs,nccmath}
\usepackage{url}
\usepackage{mathtools}
\usepackage{todonotes}
\usepackage[utf8]{inputenc}
\usepackage{calc}
\usepackage{accents}
\usepackage{stmaryrd}
\usepackage{orcidlink}

\hypersetup{
  colorlinks   = true, 
  urlcolor     = green, 
  linkcolor    =  blue, 
  citecolor   = green 
}

\newtheorem{thm}{Theorem}[section]

\newtheorem{lem}[thm]{Lemma}
\newtheorem{prop}[thm]{Proposition}
\theoremstyle{definition}
\newtheorem{defn}[thm]{Definition}
\newtheorem{rem}[thm]{Remark}
\numberwithin{equation}{section}
\newcommand{\be}{\begin{equation}}
\newcommand{\ee}{\end{equation}}
\newcommand{\bes}{\begin{equation*}}
\newcommand{\ees}{\end{equation*}}
\newcommand{\ba}{\begin{aligned}}
\newcommand{\ea}{\end{aligned}}
\newcommand{\norm}[1]{\big\Vert#1\big\Vert}

\newcommand{\sang}[1]{\langle#1\rangle}
\newcommand{\jump}[1]{\llbracket#1\rrbracket}

\newcommand{\R}{\mathbb R}
\newcommand{\T}{\mathbb{T}}
\newcommand{\p}{\partial}
\newcommand{\re}{\mathrm{e}}
\newcommand{\ri}{\mathrm{i}}
\newcommand{\rd}{\mathrm{d}}
\newcommand{\quadd}{\quad\quad}
\newcommand{\cdn}{\cdot\nabla}

\newcommand{\tr}{\mathrm{tr}\,}

\newcommand{\les}{\lesssim}

\newcommand{\trans}{^{\top}}

\newcommand{\BL}{\Big\{}
\newcommand{\BR}{\Big\}}

\newcommand{\dvg}{\nabla\cdot}
\newcommand{\ten}{\otimes}
\newcommand{\Lap}{\Delta}

\newcommand{\cur}{\nabla\times}

\newcommand{\wb}[1]{\overline{#1}}

\newcommand{\wh}[1]{\widehat{#1}}

\newcommand{\al}{\alpha}
\newcommand{\ga}{\gamma}

\newcommand{\vphi}{\varphi}
\newcommand{\om}{\omega}
\newcommand{\Ga}{\Gamma}
\newcommand{\Om}{\Omega}
\newcommand{\ka}{\kappa}
\newcommand{\Lam}{\Lambda}

\newcommand{\veps}{\varepsilon}
\newcommand{\lam}{\lambda}
\newcommand{\Ups}{\Upsilon}

\newcommand{\Hf}{\mathcal{H}}
\newcommand{\Gf}{\mathcal{G}}

\newcommand{\cE}{\mathcal{E}}
\newcommand{\cF}{\mathcal{F}}

\newcommand{\bp}{\overline{\partial}}
\newcommand{\bx}{\overline{x}}

\newcommand{\Ompm}{\Omega^{\pm}}
\newcommand{\upm}{u^{\pm}}

\newcommand{\ppm}{^{\pm}}

\newcommand{\cN}{\mathcal{N}}
\newcommand{\cM}{\mathcal{M}}

\newcommand{\fa}{\mathfrak{a}}

\newcommand{\TT}{\mathbb{T}^2}
\newcommand{\iTT}{\int_{\mathbb{T}^2}}
\newcommand{\iT}{\int_{\mathbb{T}^2}}

\newcommand{\iOm}{\int_{\Omega}}

\newcommand{\init}{_{\mathrm{in}}}

\newcommand{\Tlam}{T_{\lambda}}
\newcommand{\Tslam}{T_{\sqrt{\lambda}}}

\begin{document}
\title[Nonlinear stability of entropy waves]{Nonlinear stability of entropy waves for the Euler equations}

\author[WANG]{Wei Wang}
\address{School of Mathematical Sciences,
Zhejiang University, Hangzhou, People’s Republic of China}
\email[W. Wang]{wangw07@zju.edu.cn}

\author[ZHANG]{Zhifei Zhang}
\address{School of Mathematical Sciences, Peking University, Beijing, People’s Republic of China}
\email[Z. Zhang]{zfzhang@math.pku.edu.cn}

\author[ZHAO]{Wenbin Zhao 				  \orcidlink{0000-0003-0804-1934}
	$^{\dagger}$}
\address{School of Mathematical Sciences, Peking University, Beijing, People’s Republic of China}
\email[W. Zhao]{wenbizhao2@pku.edu.cn}
\thanks{$\dagger$Corresponding author.}

\thanks{W. Wang is supported by  NSF of China (Nos.11931010, 12271476). Z. Zhang is supported by  NSF of China (Nos. 12171010, 12288101).  W. Zhao is supported by China Postdoctoral Science Foundation (Nos.2020TQ001, 2021M690225). }

\subjclass[2010]{35Q31, 35Q35, 35R35, 76B03, 76N10}

\keywords{compressible Euler, free boundary problems, contact discontinuity, entropy wave}

\date{\today}

\begin{abstract}
In this article, we consider a class of the contact discontinuity for the full compressible Euler equations, namely the entropy wave, where the velocity is continuous across the interface while the density and the entropy can have jumps. 
The nonlinear stability of entropy waves is a longstanding open problem in multi-dimensional hyperbolic conservation laws. The rigorous treatments are challenging due to the characteristic discontinuity nature of the problem (G.-Q. Chen and Y.-G. Wang in \textit{Nonlinear partial differential equations}, Volume 7 of \textit{Abel Symp.}(2012)). In this article, we discover that the Taylor sign condition plays an essential role in the nonlinear stability of entropy waves. 
By deriving the evolution equation of the interface in the Eulerian coordinates, we relate the Taylor sign condition to the hyperbolicity of this evolution equation, which reveals a stability condition of the entropy wave. With the optimal regularity estimates of the interface, we can derive the a priori estimates  without loss of regularity.

\end{abstract}

\maketitle


\section{Introduction}

\subsection{Presentation of the problem}

In this article, we consider the full compressible Euler equations:
\begin{equation}\label{eq: system: Euler: rho u S}
\begin{cases}
	\p_t\rho+\dvg (\rho u)=0, \\
	\p_t(\rho u)+\dvg(\rho u\ten u)+\nabla p=0, \\
	\p_t(\rho\frac{|u|^2}{2}+\rho e)
	+\dvg\BL u(\rho\frac{|u|^2}{2}+\rho e+p)\BR=0,
\end{cases}
\end{equation}
where $ \rho,\,u,\,p=p(\rho,\,S),\,e=e(\rho,\,S) $ are the density, velocity, pressure and internal energy with the entropy $S$, respectively. The pressure $ p=p(\rho,\,S) $ satisfies the following state equation
\begin{equation}\label{eq: state eq}
	p=p(\rho,\,S)=A\rho^{\gamma}\re^S,
\end{equation}
with $ A>0 $ and $ \ga>1 $ as constants.

For a piecewise smooth weak solution of \eqref{eq: system: Euler: p u S} in a domain $ \Om=\T^2\times(-1,\,1) $ with
\begin{equation*}
	\Ompm=\{x=(\bx,\,x_3)=(x_1,\,x_2,\,x_3)\in\Om:\,x_3\gtrless f(t,\,\bx)\}, 
\end{equation*}
the Rankine–Hugoniot (RH) conditions are satisfied on the interface $ \Ga_f=\{(\bx,\,f(t,\,\bx):\,\bx\in\TT)\} $:
\begin{equation}\label{eq: RH for full Euler}
    \begin{cases}
    \jump{m_N}=0, \\
    m_N\jump{u_N}+|N|^2\jump{p}=0, \\
    m_N\jump{u_{\tau}}=0,\\
    m_N\jump{\frac{|u|^2}{2}+e}+\jump{pu_N}=0,
    \end{cases}
\end{equation}
where
\begin{equation}\label{def: vectors}
    N=(-\p_1f,\,-\p_2f,\,1)\trans,
    \quadd
    \tau_1=(1,\,0,\,\p_1f)\trans,
    \quadd
    \tau_2=(0,\,1,\,\p_2f)\trans,
\end{equation}
\begin{equation}\label{def: uN utau}
    u_N=u\cdot N,
    \quadd
    u_{\tau}=(u\cdot\tau_1,\,u\cdot\tau_2)\trans,
    \quadd
    m_N=\rho(u_N-\p_tf).
\end{equation}
There are two kinds of characteristic discontinuities on which $\jump{p}=\jump{u_N}=m_N=0$ (see \cite{MR2284507,MR1942469,MR3289359,MR3468916}):
\begin{itemize}
\item Vortex sheets:
\begin{equation*}
    \jump{u_{\tau}}\neq0;
\end{equation*}

\item Entropy waves:
\begin{equation*}
    \jump{u_{\tau}}=0, \quadd \jump{\rho}\neq0,
    \quadd \jump{S}\neq 0.
\end{equation*}
\end{itemize}

In this article, we focus on the entropy waves. The system \eqref{eq: system: Euler: rho u S} can be rewritten as a symmetric hyperbolic system of $ (p\ppm,\,u\ppm,\,S\ppm) $ in $\Om\ppm$ respectively:
\begin{equation}\label{eq: system: Euler: p u S}
\begin{cases}
	\frac{1}{\ga p\ppm}D_t\ppm p\ppm+\dvg u\ppm=0, \\
	\rho\ppm D_t\ppm u\ppm+\nabla p\ppm=0, \\
	D_t\ppm S\ppm=0,
\end{cases}
\end{equation}
where $D_t\ppm=\p_t+u\ppm\cdn$. Since $ \jump{u}=0 $, we shall just use $ D_t=\p_t+u\cdn $ to denote the material derivative. 
The density $ \rho\ppm $ can be recovered from \eqref{eq: state eq} by
\begin{equation}\label{eq: rho by p and S}
	\rho\ppm=A^{-\frac{1}{\ga}}(p\ppm)^{\frac{1}{\ga}}\re^{-\frac{S\ppm}{\ga}}.
\end{equation}
On the interface $ \Ga_f $, the RH conditions across $ \Ga_f $ are
\begin{equation}\label{eq: RH condition}
	\jump{p}=p^+-p^-=0, \quadd
	\jump{u}=u^+-u^-=0, \quadd
	\text{on $ \Ga_f $}.
\end{equation}
Meanwhile, the density and the entropy can have jumps across $ \Ga_f $. We shall assume that
\begin{equation}\label{eq: jump}
	\jump{\rho}=\rho^+-\rho^-\neq0,
	\quadd
	\jump{S}=S^+-S^-\neq0, \quadd
	\text{on $ \Ga_f $}.
\end{equation}

The evolution of the interface $ \Ga_f $ is given by
\begin{equation}\label{eq: f}
	\p_tf=\upm\cdot N.
\end{equation}
We shall also use $ N\ppm=\mp N $ to indicate outer normal directions of $\Ga_f$ in $ \Om\ppm $ respectively.

On the fixed upper and lower boundaries $ \Ga\ppm=\T^2\times\{\pm1\} $, there holds that
\begin{equation}\label{eq: bdry: Ga pm}
	\upm\cdot n\ppm\big|_{\Ga\ppm}=0,
\end{equation}
with $ n\ppm=(0,\,0,\,\pm1)\trans $.

In this article, we shall prove the  a priori estimates of the problem \eqref{eq: system: Euler: p u S}--\eqref{eq: bdry: Ga pm} without loss of regularity under the Taylor sign condition 
\begin{equation}\label{eq: Taylor}
	\jump{\nabla_Np}
	=\nabla_Np^+-\nabla_N p^-
	=-\nabla_{N^+}p^+-\nabla_{N^-}p^-
	>0.
\end{equation}
More precise statement of the main result will be presented in Section \ref{sec: reformulaton}. 

\subsection{History and related works}

There are three fundamental waves in the multi-dimensional hyperbolic conservation laws: shock waves, rarefaction waves and contact discontinuities (including vortex sheets and entropy waves). The interested reader is referred to \cite{MR2284507, MR3468916} for detailed discussion. The nonlinear stability of shock waves and rarefaction waves were proved in \cite{MR699241, MR683422} and \cite{MR976971, MR997387} respectively. As for contact discontinuities, they are characteristic discontinuities and usually subject to the Kelvin–Helmholtz instability and the Rayleigh-Taylor instability (see \cite{MR947785, MR2963789}). If $ \jump{u}\neq 0 $,  the contact discontinuity is also called the vortex sheet. The 3D vortex sheets are violently unstable while the 2D vortex sheets are weakly stable under the supersonic condition (see \cite{MR2159807, MR154509, MR97930}). The nonlinear stability of the 2D vortex sheets were proved in \cite{MR2095445, MR2423311} (see also \cite{MR2441089, MR3912754}). 

If $ \jump{u}=0 $ where the velocity is continuous across the interface, the contact discontinuity is also called the entropy wave. 
The normal modes analysis shows that the entropy wave is only weakly stable (see \cite{MR1942469}). Recently in \cite{MR4268832}, the authors proved the stability of the entropy wave with constant states by vanishing viscosity. However, the nonlinear stability of the general entropy waves in multi-dimensional situations is a longstanding open problem (see \cite{MR3289359}). There even lacks stability conditions addressing this problem. As stated in \cite{MR3289359}, \textit{``it would be interesting to analyze entropy waves to explore new phenomena and features of these waves in
two-dimensions and even higher dimensions."} 

 In this article, we discover that the Taylor sign condition is essential to the nonlinear stability of entropy waves and prove the a priori estimates of the problem without loss of regularity. More precisely: 
\begin{itemize}
	\item We shall derive the evolution equation of the interface and study the problem in the Eulerian coordinates. This approach was first used in \cite{MR3745155} to investigate the stability of the incompressible current-vortex sheets. It was generalized by the authors to the one-phase compressible Euler equations in \cite{Wang-Zhang-Zhao-Mach}, where it is vacuum in $ \Om^+ $. The entropy wave can be seen as a two-phase problem, where we have two fluids $ (p\ppm,\,u\ppm,\,S\ppm) $ in $ \Om\ppm $ respectively.
	
	\item  With the evolution equation of the interface, we discover that the Taylor sign condition is a natural stability condition since it is equivalent to the hyperbolicity of this evolution equation. Then we can derive the optimal regularity estimates of the interface. This enables us to investigate the quantities inside $ \Om\ppm $ in a simpler way. We do not need to make a change of coordinates or use the Alinhac's good unknowns. The a priori estimates can be derived without loss of regularity. This is important since loss of regularity is a common phenomenon for characteristic discontinuities (see \cite{MR2284507}).

	\item The Taylor sign condition is commonly used when treating the free boundary problems. However, for the two-phase compressible flow without gravity or other forces, it is a strong requirement that the Taylor sign condition \eqref{eq: Taylor} holds at each point on the interface. On the other hand, the violation of the Taylor sign condition will lead to the Rayleigh-Taylor instability (see \cite{MR2963789}).

	\item In order for the piecewise smooth solution to be the entropy wave defined in \eqref{eq: RH condition}--\eqref{eq: jump}, there need some high order compatibility conditions \eqref{eq: compatibility} on the interface. As discussed in Remark \ref{rem: compatibility}, the violation of the compatibility conditions could transform the entropy wave to a vortex sheet and lead to the Kelvin–Helmholtz instability. This suggests that the entropy wave is a really special class of contact discontinuities of the compressible Euler equations. See also the discussion of viscous contact discontinuities in \cite{MR2208289,MR2450610}. 
	
	
\end{itemize}

There is also a huge literature investigating other stabilizing effects on the interface. When taking magnetic fields under consideration, there are more types of characteristic discontinuities (see \cite{MR1942469,MR0121049}). If the magnetic fields are parallel to the interface, the characteristic discontinuities are called current-vortex sheets. The stability of the current-vortex sheets were investigated in \cite{MR2372810, MR2187618, MR2481071, MR3035981} for the compressible case and in \cite{MR2892470, MR4342131, MR2440346, MR3745155, MR2134458, Zhao-2023} for the incompressible case. If the magnetic fields are continuous across the interface and not parallel to the interface, the characteristic discontinuities are called MHD contact discontinuities. In \cite{MR3306348, MR3766987} Morando et al proved the nonlinear stability assuming that the Taylor sign condition holds. See also \cite{MR4367917} for the case with surface tension. Recently, Wang and Xin in \cite{Wang-Xin-2022} managed to prove the nonlinear stability without the assumption of the Taylor sign condition. They used the Lagrangian coordinates and verified that the normal component of the magnetic field can stabilize the interface (see also \cite{MR4340933} for the incompressible case with surface tension).

The rest of the article is organized as follows. After laying out some preliminaries in Section \ref{sec: preliminaries}, we shall derive the evolution equation of the interface and present the main result in Section \ref{sec: reformulaton}. We prove some basic estimates in Section \ref{sec: basic}. The evolution of the interface is estimated in Section \ref{sec: f}. The estimates of the pressure, the velocity and the entropy are derived in Section \ref{sec: full p u}. 
In Appendix \ref{sec: elliptic},  some results on the elliptic systems are presented. We also list some analytic tools in Appendix \ref{sec: paradiff}.

\section{Preliminaries}\label{sec: preliminaries}

In this section, we recall some preliminary results on the harmonic coordinates and the Dirichlet-Neumann (DN) operators from \cite{MR3745155, Wang-Zhang-Zhao-Mach}. The notations and basic properties of paradifferential operators are included in Appendix \ref{sec: paradiff}.

\subsection{Harmonic coordinates}
Given a smooth function $ f_{\ast}=f_{\ast}(\bx) $, we define a reference domain
\begin{equation*}
	\Om_{\ast}\ppm=\{x\in\Om:\,x_3\gtrless f_{\ast}(\bx)\},
	\quadd
	\Ga_*=\{(\bx,\,f_*(\bx)):\,\bx\in\TT\}.
\end{equation*}

We shall consider the free boundary problem that lies in a neighborhood of the reference domain $ \Om_{\ast} $. To this end, we define
\begin{equation*}
	\mathcal{N}(\delta,\,\ka)
	=\{f\in H^{\ka}(\TT):\,\norm{f-f_{\ast}}_{H^{\ka}(\TT)}\leq\delta\}.
\end{equation*}
For a function $ f\in\mathcal{N}(\delta,\,\ka) $, set
\begin{equation*}
	\Om_f\ppm=\{x\in\Om:\,x_3\gtrless f(\bx )\},
	\quadd
	\Ga_f=\{(\bx,\,f(\bx )):\,\bx\in\TT\}.
\end{equation*}
Then we can introduce the harmonic coordinates. Define $ \Phi_f\ppm:\,\Om_{\ast}\ppm\to\Om_f\ppm $ by the harmonic extension:
\begin{equation}\label{def: harmonic coordinates}
\begin{cases}
	\Lap\Phi_f\ppm=0, & \text{in $ \Om_{\ast}\ppm $}, \\
	\Phi_f\ppm(\bx ,\,f_{\ast}(\bx))=(\bx ,\,f(\bx )), & \text{on $ \Ga_{\ast} $}, \\
	\Phi_f\ppm(\bx ,\,\pm1)=(\bx ,\,\pm1), & \text{on $ \Ga\ppm $}.
\end{cases}
\end{equation}
Given $ f_{\ast} $, there exists $ \delta_0>0 $ such that $ \Phi_f\ppm $ is bijective when $ \delta\in[0,\,\delta_0] $. Thus we can also define the inverse map $ (\Phi_f\ppm)^{-1}:\,\Om_f\ppm\to\Om_{\ast}\ppm $ such that
\begin{equation*}
	\Phi_f\ppm\circ(\Phi_f^{\pm})^{-1}=\mathrm{Id},
	\quadd
	(\Phi_f^{\pm})^{-1}\circ\Phi_f\ppm=\mathrm{Id}.
\end{equation*}

Let us list some basic inequalities about harmonic coordinates without proof.
\begin{lem}\label{lem: harmonic coord.}
	Assume that $ f\in\mathcal{N}(\delta_0,\,\ka) $ with $ \ka\geq 4 $. Then there exists a constant $ C=C(\delta_0,\,\norm{f_{\ast}}_{H^{\ka}(\TT)}) $ such that
	\begin{enumerate}
		\item If $ u\in H^{s}(\Om_f\ppm) $ with $ s\in[0,\,\ka] $, then
		\begin{equation*}
			\norm{u\circ\Phi_f\ppm}_{H^{s}(\Om_{\ast}\ppm)}
			\leq C\norm{u}_{H^s(\Om_{f}\ppm)}.
		\end{equation*}
	
		\item If $ u\in H^{s}(\Om_{\ast}\ppm) $ with $ s\in[0,\,\ka] $, then
		\begin{equation*}
			\norm{u\circ(\Phi_f^{\pm})^{-1}}_{H^{s}(\Om_{f}\ppm)}
			\leq C\norm{u}_{H^s(\Om_{\ast}\ppm)}.
		\end{equation*}
	
		\item If $ u,\,v\in H^{s}(\Om_f\ppm) $ with $ s\in[2,\,\ka] $, then
		\begin{equation*}
			\norm{uv}_{H^{s}(\Om_f\ppm)}
			\leq C\norm{u}_{H^{s}(\Om_f\ppm)}
			\norm{v}_{H^{s}(\Om_f\ppm)}.
		\end{equation*}
	\end{enumerate}
	
\end{lem}

\subsection{The Dirichlet-Neumann operator}

For a smooth enough function $ g=g(\bx ) $ on $ \Ga_f=\{(\bx ,\,f(\bx )):\,\bx \in\T^2\} $, denote the harmonic extension of $ g $ to $ \Om_f\ppm $ by $ \Hf_f\ppm g $, that is,
\begin{equation}\label{def: harmonic extension}
\begin{cases}
	\Lap\Hf_f\ppm g=0,
		& \text{in $ \Om_f\ppm $}, \\
	(\Hf_f\ppm g)(\bx ,\,f(\bx ))=g(\bx ),
		& \text{on $ \Ga_f $}, \\
	(\Hf_f\ppm g)(\bx,\,\pm 1)=0,
		& \text{on $ \Ga\ppm $}.
\end{cases}
\end{equation}
Here we use the Dirichlet boundary condition on the bottom $ \Ga\ppm $ instead of the Neumann boundary condition as in the usual case. This modification is useful in the energy estimates in the following sections.

For a smooth enough function $ g=g(\bx) $ on $ \Ga_f=\{(\bx,\,f(\bx)):\,\bx\in\TT \} $, define
\begin{equation}\label{def: DN op}
	\Gf_f\ppm g
	=N\ppm\cdn\Hf_f\ppm g\big|_{\Ga_f}
	=\mp N\cdn\Hf_f\ppm g\big|_{\Ga_f},
\end{equation}
where $ N=(-\p_1f,\,-\p_2f,\,1)\trans $ is the scaled normal vector on the surface $ \Ga_f $. However, all the regularity properties of the Dirichlet-Neumann operator will be kept in spite of the modification, as discussed in the appendix of \cite{Wang-Zhang-Zhao-Mach}. The same arguments in \cite{MR3060183} yield the following basic properties of the DN operator.
\begin{lem}\label{lem: DN op}
	Let $ f\in\mathcal{N}(\delta_0,\,\ka) $ with $ \ka\geq 4 $. Then there exists a constant $ C=C(\delta_0,\,\norm{f}_{H^{\ka}(\TT)}) $ such that
	\begin{enumerate}
		\item $ \Gf_f\ppm $ is self-adjoint:
		\begin{equation*}
			(\Gf_f\ppm\phi,\,\psi)=(\phi,\,\Gf_f\ppm\psi),
			\quadd
			\forall\phi,\,\psi\in H^{\frac{1}{2}}(\TT).
		\end{equation*}
		
		\item $ \Gf_f\ppm $ is positive:
		\begin{equation*}
			(\Gf_f\ppm\phi,\,\phi)
			\geq C\norm{\phi}_{\dot{H}^{\frac{1}{2}}(\TT)}.
		\end{equation*}

	\end{enumerate}
\end{lem}

By the appendix of \cite{MR3745155}, we also have the following paralinearization of the DN operators.
\begin{lem}
	Assume that $ f\in H^{\ka}(\TT) $ with $ \ka\geq4 $. 
	Then the DN operators $ \Gf_f\ppm $ can be decomposed as
	\begin{equation}\label{eq: decom of Gf}
		\Gf_f\ppm=T_{\lam}+R_f\ppm,
	\end{equation}
	where the symbol of the leading term is
	\begin{equation}\label{def: lambda}
		\lam(x,\,\xi)=\sqrt{(1+|\nabla f|^2)|\xi|^2
		-(\nabla f\cdot \xi)^2}
	\end{equation}
	and the remainder terms $ R_f\ppm $ satisfy that
	\begin{equation}\label{estim: Rf}
		\norm{R_f\ppm g}_{H^{s}}
		\leq C\big(\norm{f}_{H^{\ka}}\big)
		\norm{g}_{H^s},
		\quadd\forall s\in[1/2,\,\ka-1].
	\end{equation}
	Furthermore, there holds that
	\begin{equation}\label{estim: Gf}
		\norm{\Gf_f\ppm g}_{H^{s-1}}
		\leq C\big(\norm{f}_{H^{\ka}}\big)
		\norm{g}_{H^s},
		\quadd\forall s\in[1/2,\,\ka].
	\end{equation}
\end{lem}

\subsection{Notations}

When there is no ambiguity, we shall use 
$ u(x)=\mathbf{1}_{\Om^+}(x)u^+(x)+\mathbf{1}_{\Om^-}(x)u^-(x) $ and $ \norm{u}_{H^s(\Om)}=\norm{u^+}_{H^s(\Om^+)}+\norm{u^-}_{H^s(\Om^-)} $ to simplify notations.

Since $ u $ is continuous across $ \Ga_f $ from \eqref{eq: RH condition}, we shall just use $ D_t=\p_t+u\cdn $ as the material derivative. The tangential derivatives in $\Om\ppm$ are
\begin{equation}\label{def: tangential derivatives}
\ba
	\bp_t = \p_t +\Hf_f\ppm(\p_t f)\p_3 , \quadd
	\bp_j =  \p_j +\Hf_f\ppm(\p_j f)\p_3 	\quad (j=1,\,2).
\ea
\end{equation}
By \eqref{eq: f}, it is direct to verify that
\begin{equation}\label{eq: Dt}
	D_t=\bp_t+u_1\bp_1+u_2\bp_2,
	\quadd \text{on $ \Ga_f $}.
\end{equation}
By the definition of the harmonic extension \eqref{def: harmonic extension}, the derivatives $ \bp=(\bp_1,\,\bp_2) $ are tangential to both $ \Ga_f $ and $ \Ga\ppm $. We denote by
\begin{equation}\label{def: Lam, Ums}
	\Lam=\sang{\nabla}=(1-\Lap)^{1/2},
	\quadd
	\Ups=\sang{\bp}=(1+|\bp|^2)^{1/2}
\end{equation}
when treating high order derivatives.

The minuscule indices $ i,\,j,\,k $ are in $ \{1,\,2\} $ and the capital indices $ J,\,K $ are in $ \{1,\,2,\,3\} $. We shall use the Einstein summation convention, that is, a repeated index in a term means summation of terms over the index. To simply the arguments, we also omit all the binomial coefficients $ \binom{l}{m}=\frac{l!}{m!(l-m)!} $. 

In the following, we shall use $ \Hf\ppm=\Hf_f\ppm $ and $ \Gf\ppm=\Gf_f\ppm $ if there is no confusion of the function $ f $.

\section{Reformulation and main result}\label{sec: reformulaton}

\subsection{Evolution of interface $ f $}\label{subsec: evollution of f}

By \eqref{eq: Dt}, we can rewrite the evolution equation of $ f $ in \eqref{eq: f} as
\begin{equation}\label{eq: f: Dt f}
	D_tf=u_3\ppm.
\end{equation}
By taking another derivative $ D_t $ to both sides of \eqref{eq: f: Dt f}, we have
\begin{equation}\label{eq: f: Dt2 f}
	D_t^2f=D_tu\ppm_3.
\end{equation}
For $ i=1,\,2 $, there holds that
\begin{equation}\label{eq: f: Dt2 bpi f}
\ba
	D_t^2\bp_if
	= & \bp_iD_tu_3\ppm+[D_t^2,\,\bp_i]f \\
	= & \bp_iD_tu_3\ppm
	+[D_t,\,\bp_i]D_tf+D_t[D_t,\,\bp_i]f \\
	= & \bp_iD_tu_3\ppm-\bp_iu_j\ppm\bp_jD_tf
	-D_t(\bp_iu_j\ppm \bp_jf) \\
	= & \bp_iD_tu_3\ppm
	-\bp_iu_j\ppm(D_t\bp_jf+\bp_ju_k\ppm\bp_kf) \\
	& -(\bp_i D_tu_j\ppm-\bp_iu_k\ppm\bp_ku_j\ppm)\bp_jf
	-\bp_iu_j\ppm D_t\bp_jf \\
	= & \bp_i D_tu\ppm\cdot N-2\bp_iu_j\ppm D_t\bp_jf.
\ea
\end{equation}
Furthermore, we can plug the momentum equation in \eqref{eq: system: Euler: p u S} to \eqref{eq: f: Dt2 bpi f} to get that
\begin{equation}\label{eq: f: Dt2 bpi f: p}
\ba
	\rho\ppm D_t^2\bp_if
	= & -\bp_i\nabla p\ppm\cdot N
	+\frac{\bp_i\rho\ppm}{\rho\ppm}\nabla_N p\ppm 
	-2\rho\ppm\bp_iu_j\ppm D_t\bp_jf \\
	= & -\BL\nabla\bp_ip\ppm-\p_3p\ppm \nabla\Hf\ppm(\bp_if)\BR\cdot N
	+\frac{\bp_i\rho\ppm}{\rho\ppm}\nabla_N p\ppm 
		-2\rho\ppm\bp_iu_j\ppm D_t\bp_jf \\
	= & \p_3p\ppm\nabla_N\Hf\ppm(\bp_if)
	-\nabla_N\bp_ip\ppm 
	+\frac{\bp_i\rho\ppm}{\rho\ppm}\nabla_N p\ppm 
		-2\rho\ppm\bp_iu_j\ppm D_t\bp_jf.
\ea	
\end{equation}
Thus, by taking a sum of \eqref{eq: f: Dt2 bpi f: p} with the positive index and the negative index respectively, we have 
\begin{equation}\label{eq: f: Dt2 bpi f: +-}
	\ba
		D_t^2\bp_if
		+\frac{1}{\rho^++\rho^-}\BL\p_3p^+\Gf^+
		-\p_3p^-\Gf^-\BR\bp_if
		=\frac{1}{\rho^++\rho^-}\BL\cM^++\cM^-\BR,
	\ea
\end{equation}
with the Dirichlet-Neumann operators $ \Gf\ppm $ defined in \eqref{def: DN op} and 
\begin{equation}\label{def: cM pm}
	\cM\ppm
	=-\nabla_N\bp_ip\ppm 
	+\frac{\bp_i\rho\ppm}{\rho\ppm}\nabla_N p\ppm 
	-2\rho\ppm\bp_iu_j\ppm D_t\bp_jf.
\end{equation}
Lastly, the decomposition of the DN operators $ \Gf\ppm=T_{\lam}+R\ppm $ in \eqref{eq: decom of Gf} can be applied to get that
\begin{equation}\label{eq: evolution: f}
	D_t^2\bp_if+\fa T_{\lam}\bp_if
	= \cN^++\cN^-,
\end{equation}
where the Taylor sign $ \fa $ is
\begin{equation}\label{def: fa}
	\fa=\frac{\p_3p^+-\p_3p^-}{\rho^++\rho^-},
\end{equation}
and the lower order terms are
\begin{equation}\label{def: cN}
\ba
	\cN\ppm= & 
	-\frac{1}{\rho^++\rho^-}\nabla_N\bp_ip\ppm 
	+\frac{\bp_i\rho\ppm}{(\rho^++\rho^-)\rho\ppm}
	\nabla_Np\ppm \\
	& -\frac{2\rho\ppm}{\rho^++\rho^-}\bp_iu_j\ppm D_t\bp_jf
	\mp\frac{\p_3p\ppm}{\rho^++\rho^-}R\ppm\bp_if.
\ea
\end{equation}
Since
\begin{equation*}
	\p_3=\frac{1}{1+|\nabla f|^2}\BL\nabla_N+\p_if\bp_i\BR,
\end{equation*}
the RH conditions \eqref{eq: RH condition} and the Taylor sign condition \eqref{eq: Taylor} imply that
\begin{equation}\label{eq: fa to taylor sign}
\ba
	\fa= & \frac{1}{(\rho^++\rho^-)(1+|\nabla f|^2)}
	\BL\jump{\nabla_N p}+\p_if\jump{\bp_i p}\BR \\
	= & \frac{\jump{\nabla_Np}}{(\rho^++\rho^-)(1+|\nabla f|^2)}>0.
\ea
\end{equation}

\subsection{Evolution of the vorticity $ \om\ppm $ and the entropy  $ S\ppm $}
By taking the curl of the both sides of the momentum equation in \eqref{eq: system: Euler: p u S}, we have the evolution equation of the vorticity $ \om\ppm $ as
\begin{equation}\label{eq: evolution: om}
	D_t\om\ppm=\om\ppm\cdn u\ppm
	-\om\ppm(\dvg u\ppm)
	-\nabla\frac{1}{\rho\ppm}\times\nabla p\ppm.
\end{equation}

The evolution of the entropy $ S\ppm $ is given by the entropy equation in \eqref{eq: system: Euler: p u S} as
\begin{equation}\label{eq: evolution: S}
	D_tS\ppm=0.
\end{equation}

\subsection{Evolution of the pressure $ p\ppm $}

To derive the evolution equation of the pressure $ p\ppm $, we take $ D_t\eqref{eq: system: Euler: p u S}_1-\dvg\BL\frac{1}{\rho\ppm}\times\eqref{eq: system: Euler: p u S}_2\BR $ to get that
\begin{equation}\label{eq: wave of p}
	D_t\Big(\frac{1}{\ga p\ppm}D_t p\ppm\Big)
	-\dvg\Big(\frac{1}{\rho\ppm}\nabla p\ppm\Big)
	=\tr(\nabla u\ppm)^2.
\end{equation}
This is a second order wave equation.

\subsection{Compatibility conditions}\label{subsec: p on Gaf}

Assume that $ \ka\geq4 $ is an integer. To estimate high order derivatives of the piecewise smooth weak solutions, we need the following compatibility conditions on the interface $ \Ga_f $:
\begin{equation}\label{eq: compatibility}
	\jump{D_t^l p}=0, 
	\quadd
	\jump{D_t^l u}=0,
	\quadd
	0\leq l\leq\ka+1.
\end{equation}

Next we discuss an important consequence of the compatibility conditions \eqref{eq: compatibility}. Assume that the solution $ (p\ppm,\,u\ppm,\,S\ppm)\in H^3(\Om\ppm)\subset C^1(\Om\ppm) $ and $ f\in H^3(\TT)\subset C^1(\TT) $. Recall the vectors $ \tau_1 $ and $ \tau_2 $ in \eqref{def: vectors} which are tangential to the interface $ \Ga_f $. The RH conditions \eqref{eq: RH condition} and the compatibility conditions \eqref{eq: compatibility} can be applied to the momentum equation in \eqref{eq: system: Euler: p u S} to get that
	\begin{equation}\label{eq: jump of momentum}
		\jump{\frac{\nabla p}{\rho}}
		=-\jump{D_t u}=0.
	\end{equation}
Furthermore, since $\jump{p}=0$ and
	\begin{equation}\label{eq: jump of tangential p}
		\jump{\tau_i\cdn p}=\tau_i\cdn\jump{p}=0,
		\quadd i=1,\,2,
	\end{equation}
by \eqref{eq: jump of momentum} there holds that
	\begin{equation}\label{eq: jump of momentum in tau}
		0=\tau_i\cdot\jump{\frac{\nabla p}{\rho}}
		=\jump{\frac{1}{\rho}}(\tau_i\cdn p),
		\quadd i=1,\,2.
	\end{equation}
That is, when $ \jump{\rho}\neq 0 $, by \eqref{eq: jump of momentum in tau} we must have 
	\begin{equation}\label{eq: bpi p =0}
		\bp_ip\big|_{\Ga_f}=\tau_i\cdn p\big|_{\Ga_f}=0 \quadd (i=1,\,2),
	\end{equation}
which implies that the pressure on the interface $ p|_{\Ga_f}=p(t,\bx,f(t,\,\bx)) $ is independent on the space variables $ \bx $. Thus, we may assume that
\begin{equation}\label{def: q}
		p\big|_{\Ga_f}=q(t),
		\quadd
		q(0)=0.
\end{equation}

\begin{rem}\label{rem: compatibility}
	The consequence \eqref{eq: bpi p =0} is necessary for the piecewise smooth weak solution to be an entropy wave. If \eqref{eq: bpi p =0} fails, the tangential component of the pressure force is not trivial. Since the densities have a jump $\jump{\rho}\neq0$, the accelerations on both sides must have a jump:
	\begin{equation*}
	    \jump{D_tu}\cdot\tau_j
	    =-\jump{\frac{\tau_j\cdn p}{\rho}}
	    =-\jump{\frac{1}{\rho}}(\tau_j\cdn p)
	    \neq 0.
	\end{equation*}
	The entropy wave could evolve to a vortex sheet immediately. 
\end{rem}

\subsection{Main result}

Assume that $ \ka\geq4 $ is an integer. The full energy norm is defined as
\begin{equation}\label{def: cE}
\ba
	\cE(t)= & \norm{f}_{H^{\ka}(\TT)}
	+\norm{D_t f}_{H^{\ka-\frac{1}{2}}(\TT)} \\
	& +\norm{u}_{H^{\ka}(\Om)}
	+\norm{D_tu}_{H^{\ka-1}(\Om)}
	+\sum_{l=2}^{\ka+1}\norm{D_t^lu}_{H^{\ka+1-l}(\Om)} \\
	& +\norm{(p,\,S)}_{H^{\ka}(\Om)}
 	+\sum_{l=1}^{\ka+1}\norm{(D_t^lp,\,D_t^lS)}_{H^{\ka+1-l}(\Om)}.
\ea
\end{equation}
For some $ T>0 $, we shall denote by
\begin{equation*}
	M=\sup_{0<t<T}\cE(t), 
	\quadd
	M_0=\cE(0).
\end{equation*}
The lower order energy norm is
\begin{equation}\label{def: cF}
\ba
	\cF(t)= \norm{(p,\,u,\,S)}_{H^{\ka-1}(\Om)}
	+\sum_{l=1}^{\ka}\norm{(D_t^lp,\,D_t^lu,\,D_t^lS)}_{H^{\ka-l}(\Om)}.
\ea
\end{equation}

\begin{thm}\label{thm}
	Let $ \ka\geq4 $ be an integer. Suppose that the initial data $ (f\init,\,p\init\ppm,\,u\init\ppm,\,S\init\ppm) $ satisfy the bound $ \cE(0)=M_0<\infty $. Furthermore, assume that there are two constants $ 0<c_0<C_0 $  such that
	\begin{enumerate}
		\item $ c_0\leq\rho\init\ppm\leq C_0 $, \quad $ c_0\leq p\init\ppm\leq C_0 $;
		
		\item $ -1+c_0\leq f\init\leq 1-c_0 $;
		
		\item $ \frac{N\init}{|N\init|}\cdot\jump{\nabla p\init\ppm}\geq c_0 $.
	\end{enumerate}
	Then for $ T>0 $ small enough, the solution $ (f,\,p\ppm,\,u\ppm,\,S\ppm) $ to the problem \eqref{eq: system: Euler: p u S}--\eqref{eq: bdry: Ga pm} under the compatibility conditions \eqref{eq: compatibility} satisfies that
	\begin{enumerate}
	\item $ M\leq C(M_0,\,c_0,\,C_0)+TC(M,\,c_0,\,C_0) $;
	
	\item $ -1+\frac{c_0}{2}\leq f\leq 1-\frac{c_0}{2} $;
			
	\item $ \frac{N}{|N|}\cdot\jump{\nabla p\ppm}\geq \frac{c_0}{2} $.
	\end{enumerate}
	
\end{thm}

The constants $ C(M_0,\,c_0,\,C_0) $ are continuous functions of $ M_0,\,c_0,\,C_0 $. In the energy estimates in the following sections, we shall take $ c_0 $ and $ C_0 $ to be fixed and just use $ C(M_0) $ to denote constants from line to line.

\section{Basic energy estimates}\label{sec: basic}

In this section, we prove some basic energy estimates.

\subsection{Lower order estimates}

For the lower order energy norm $ \cF $ defined in \eqref{def: cF}, we have the following estimates.

\begin{prop}\label{prop: low order}
	For $ t\in[0,\,T] $, there holds that
	\begin{equation}\label{estim: cF}
		\cF(t)
		\leq C(M_0)+TC(M).
	\end{equation}
	Furthermore,
	\begin{equation}\label{estim: low Linfty}
		\norm{(p,\,u,\,S)}_{W^{\ka-3,\infty}(\Om)}
		+\sum_{l=1}^{\ka-2}\norm{(D_t^lp,\,D_t^lu,\,D_t^lS)}_{W^{\ka-2-l,\infty}(\Om)}
		\leq C(M_0)+TC(M).
	\end{equation}
\end{prop}

\begin{proof}
	For $ 1\leq l\leq\ka $, since
	\begin{equation*}
		\begin{cases}
			D_t\Lam^{\ka-l}D_t^{l}p
			=\Lam^{\ka-l}D_t^{l+1}p
			-[\Lam^{\ka-l},\,D_t]D_t^lp, \\
			D_t\Lam^{\ka-l}D_t^{l} u
			=\Lam^{\ka-l}D_t^{l+1}u
			-[\Lam^{\ka-l},\,D_t]D_t^lu, \\
			D_t\Lam^{\ka-l}D_t^{l}S
			=0,
		\end{cases}
	\end{equation*}
	we have 
	\begin{equation*}
		\ba
		& \frac{1}{2}\frac{\rd}{\rd t}
		\norm{(D_t^lp,\,D_t^l u,\,D_t^lS)}_{H^{\ka-l}(\Om)}^2 \\
		= & \iOm \BL D_t\Lam^{\ka-l}D_t^{l}p
		\cdot \Lam^{\ka-l}D_t^{l}p
		+D_t\Lam^{\ka-l}D_t^{l} u
		\cdot \Lam^{\ka-l}D_t^{l} u
		+D_t\Lam^{\ka-l}D_t^{l} S
		\cdot \Lam^{\ka-l}D_t^{l} S\BR\rd x \\
		& +\iOm\frac{\dvg u}{2}
		\BL|\Lam^{\ka-l}D_t^{l} p|^2
		+|\Lam^{\ka-l}D_t^{l}u|^2
		+|\Lam^{\ka-l}D_t^{l} S|^2\BR\rd x \\
		\leq & C(M). 
		\ea
	\end{equation*}
	The case when $ l=0 $ follows in a similar way. 
	
	Application of the Sobolev inequalities to \eqref{estim: cF} proves \eqref{estim: low Linfty}.
	
\end{proof}

\subsection{Tangential energy of $ (p,\,u) $}

Recall the equations for the pressure and velocity $ (p,\,u) $ in \eqref{eq: system: Euler: p u S}:
\begin{equation}\label{eq: (p,u)}
	\begin{cases}
		\frac{1}{\ga p}D_t p+\dvg u=0, \\
		\rho D_t u+\nabla p=0.
	\end{cases}
\end{equation}
For $ 0\leq l\leq\ka+1 $, by taking $ D_t^l $ to both sides of \eqref{eq: (p,u)}, we have the system for $ (D_t^lp,\,D_t^lu) $ as
\begin{equation}\label{eq: (p,u) high}
	\begin{cases}
		\frac{1}{\ga p}D_t^{l+1}p+\dvg D_t^lu
		=\cN_p^l, \\
		\rho D_t^{l+1} u+\nabla D_t^l p
		=\cN_u^l,
	\end{cases}
\end{equation}
where
\begin{equation*}
	\cN_p^l=[\frac{1}{\ga p},\,D_t^{l}]D_tp
		+[\dvg,\,D_t^{l}]u,
	\quadd
	\cN_u^l=[\rho ,\,D_t^{l}]D_tu
		+[\nabla,\,D_t^l]p.
\end{equation*}

\begin{prop}\label{prop: p u tangential}
	For $ t\in[0,\,T] $, there holds that
	\begin{equation}\label{estim: Dtl(p, u)}
		\sum_{l=0}^{\ka+1}\norm{(D_t^lp,\,D_t^lu)}_{L^2(\Om)}
		\leq C(M_0)+TC(M).
	\end{equation}
\end{prop}

\begin{proof}

Since $ \ka\geq4 $ and
\begin{equation*}
\ba
	\norm{(\cN_p^l,\,\cN_u^l)}_{L^2(\Om)}
	\leq & C(M),
\ea
\end{equation*}
energy estimates of \eqref{eq: (p,u) high} yield that
\begin{equation}\label{ineq: tangential energy}
\ba
	& \frac{\rd}{\rd t}\iOm\BL\frac{1}{\ga p}\frac{|D_t^l p|^2}{2}+\rho\frac{|D_t^lu|^2}{2}\BR\rd x \\
	= & \iOm D_t\BL\frac{1}{\ga p}\frac{|D_t^lp|^2}{2}+\rho\frac{|D_t^lu|^2}{2}\BR\rd x
	+\iOm(\dvg u)\BL\frac{1}{\ga p}\frac{|D_t^lp|^2}{2}+\rho\frac{|D_t^lu|^2}{2}\BR\rd x \\
	= & \iOm \BL D_t\big(\frac{1}{\ga p}\big)\frac{|D_t^lp|^2}{2}+D_t\rho\frac{|D_t^lu|^2}{2}\BR\rd x
	+\iOm(\dvg u)\BL\frac{1}{\ga p}\frac{|D_t^lp|^2}{2}+\rho\frac{|D_t^lu|^2}{2}\BR\rd x \\
	& + \iOm\BL D_t^lp\cN_p^l+D_t^lu\cdot\cN_u^l\BR\rd x
	+\iOm\dvg(D_t^luD_t^lp)\rd x \\
	\leq & C(M),
\ea
\end{equation}
where we have used the boundary conditions \eqref{eq: bdry: Ga pm} and the compatibility conditions \eqref{eq: compatibility} .

\end{proof}

\subsection{Estimates of $ \omega $ and $ S $}

From \eqref{eq: evolution: om}--\eqref{eq: evolution: S}, the vorticity $ \om $ and the entropy $ S $ satisfy the following transport equations respectively:
\begin{equation}\label{eq: om}
	D_t\om
	=\om\cdn u-\om\dvg u
	-\nabla\frac{1}{\rho}\times\nabla p,
\end{equation}
and
\begin{equation}\label{eq: S}
	D_tS=0.
\end{equation}
Direct energy estimates yield the following result.
\begin{prop}\label{prop: om & S}
	For $ t\in[0,\,T] $, there hold that
	\begin{equation}\label{estim: om}
		\norm{\om}_{H^{\ka-1}(\Om)}
		+\sum_{l=1}^{\ka}\norm{D_t^l\om}_{H^{\ka-l}(\Om)}
		\leq M_0+T C(M),
	\end{equation}
	and
	\begin{equation}\label{estim: S}
			\norm{S}_{H^{\ka}(\Om)}
			+\sum_{l=1}^{\ka+1}\norm{D_t^lS}_{H^{\ka+1-l}(\Om)}
			\leq M_0+T C(M).
		\end{equation}
	
\end{prop}

\begin{proof}
	Taking $ \Lam^{\ka-l}D_t^l $ $ (1\leq l\leq\ka) $ to both sides of \eqref{eq: om}, we have
	\begin{equation}\label{eq: om high}
		D_t\Lam^{\ka-l}D_t^l\om
		=[D_t,\,\Lam^{\ka-l}]D_t^l\om
		+\Lam^{\ka-l}D_t^l(\om\cdn u-\om\dvg u
		-\nabla\frac{1}{\rho}\times\nabla p).
	\end{equation}
	Then, the Sobolev inequalities and \eqref{eq: om high} can be applied to get that
	\begin{equation*}
		\ba
			& \frac{1}{2}\frac{\rd}{\rd t}\iOm|\Lam^{\ka-l}D_t^l\om|^2\rd x \\
			= & \iOm D_t\Lam^{\ka-l}D_t^l\om\cdot
			\Lam^{\ka-l}D_t^l\om\rd x
			+\iOm\frac{\dvg u}{2}|\Lam^{\ka-l}D_t^l\om|^2\rd x \\
			\leq & C(M)\norm{D_t\Lam^{\ka-l}D_t^l\om}_{L^2(\Om)}
			+C(M)\norm{\dvg u}_{L^{\infty}({\Om})} \\
			\leq & C(M).
		\ea
	\end{equation*}
	This proves \eqref{estim: om} with $ 1\leq l\leq\ka+1 $. The case when $ l=0 $ follows similarly. 
	
	For the entropy $ S $ in \eqref{eq: S}, since
	\begin{equation*}
		D_t\Lam^{\ka+1-l}D_t^lS=[D_t,\,\Lam^{\ka+1-l}]D_t^lS, 
	\end{equation*}
	the estimates in \eqref{estim: S} can be proved just as those in \eqref{estim: om}. 
\end{proof}

\section{Estimates of the interface $ f $}\label{sec: f}
In this section, we shall derive the estimates of the interface $ f $.

\begin{prop}\label{prop: f}
	For $ t\in[0,\,T] $, there holds that
	\begin{equation}\label{estim: f}
		\norm{f}_{H^{\ka}(\TT)}
		+\norm{D_t f}_{H^{\ka-\frac{1}{2}}(\TT)}
		\leq M_0+TC(M).
	\end{equation}
\end{prop}

Recall the equation of $ f $ in \eqref{eq: evolution: f}:
\begin{equation}\label{eq: f section}
	D_t^2\bp_if+\fa T_{\lam}\bp_if
	= \cN^++\cN^-,
\end{equation}
where $ \lam $, $ \fa $, and $ \cN\ppm $ are given by \eqref{def: lambda}, \eqref{def: fa}, and \eqref{def: cN} respectively. The rest of the section is devoted to the proof of Proposition \ref{prop: f}.

Set
\begin{equation*}
	F =\Ups^{\ka-\frac{3}{2}}\wb{\p}_if
	=\sang{\bp}^{\ka-\frac{3}{2}}\wb{\p}_if.
\end{equation*}
Taking $ \Ups^{\ka-\frac{3}{2}} $ to both sides of \eqref{eq: f section}, we have the equation for $  F  $ as
\begin{equation}\label{eq: Dt2 pi f: high order}
	\begin{split}
		D_t^2 F
		+\fa T_{\lam} F
		= &
		-[\Ups^{\ka-\frac{3}{2}},\,D_t^2]  \wb{\p}_if
		- [\Ups^{\ka-\frac{3}{2}},\,\fa T_{\lam}]  \wb{\p}_if
		+\Ups^{\ka-\frac{3}{2}}(\cN^++\cN^-).	
	\end{split}
\end{equation}
Direct computation shows that
\begin{equation*}
	\begin{split}
		& \frac{1}{2}\frac{\rd}{\rd t}
		\iT\BL|D_t F |^2+\fa|\Tslam F |^2\BR \rd\bx \\
		= & \frac{1}{2}\iT D_t
		\BL|D_t F |^2+\fa|\Tslam F |^2\BR \rd\bx
		+\frac{1}{2}\iT(  \wb{\p}_ju_j)
		\BL|D_t F |^2+\fa|\Tslam F |^2\BR \rd\bx \\
		= & \iT\BL D_t^2 F \cdot D_t F
		+\fa\cdot D_t\Tslam F
		\cdot \Tslam F \BR \rd\bx \\
		& + \frac{1}{2}\iT\BL( \wb{\p}_ju_j)|D_t F |^2
		+(\wb{\p}_ju_j\fa+D_t\fa)
		|\Tslam F |^2 \BR\rd\bx \\
		= & \iT\BL D_t^2 F + \fa\Tlam F \BR  D_t  F  \rd\bx
		+\iT \fa \BL\Tslam^*\Tslam F-\Tlam F\BR D_tF\rd\bx \\
		& + \iT [\fa D_t,\,\Tslam] F
		\cdot\Tslam F  \rd\bx
		+\frac{1}{2}\iT\BL(\wb{\p}_ju_j)|D_t F |^2
		+(\wb{\p}_ju_j\fa+D_t\fa)
		|\Tslam F |^2\BR \rd\bx.
	\end{split}
\end{equation*}
Therefore,
\begin{equation}\label{eq: dt E1}
	\begin{split}
		& \frac{1}{2}\frac{\rd}{\rd t}
		\BL|D_t F |^2+\fa|\Tslam F |^2\BR \rd\bx \\
		= & \frac{1}{2}\iT\BL(\wb{\p}_ju_j)|D_t F |^2
		+(\wb{\p}_ju_j\fa+D_t\fa)
		|\Tslam F |^2\BR \rd\bx \\
		& +\iT \fa \BL\Tslam^*\Tslam F-\Tlam F\BR D_tF\rd\bx 
		+ \iT [\fa D_t,\,\Tslam] F
		\cdot\Tslam F  \rd\bx \\
		& -\iT [\Ups^{\ka-\frac{3}{2}},\,D_t^2]  \wb{\p}_if
		\cdot D_tF\rd\bx 
		-\iT [\Ups^{\ka-\frac{3}{2}},\,\fa T_{\lam}]  \wb{\p}_if
		\cdot D_tF\rd\bx \\
		& +\iT\Ups^{\ka-\frac{3}{2}}(\cN^++\cN^-)
		\cdot D_tF\rd\bx \\
		:= & I_1+I_2+I_3+I_4+I_5+I_6.
	\end{split}
\end{equation}

For $ I_1 $ in \eqref{eq: dt E1}, it is direct to verify that
\begin{equation}\label{ineq: I1}
	I_1\leq C(M)\BL\norm{D_tF}_{L^2(\TT)}^2
	+\norm{\Tslam F}_{L^2(\TT)}^2\BR
	\leq C(M).
\end{equation}
For $ I_2 $ in \eqref{eq: dt E1}, it follows from Lemma \ref{lem: symbol calculus} that
\begin{equation}\label{ineq: I2}
	I_2\leq C(M)
	\norm{(\Tslam^*\Tslam-\Tlam)F}_{L^2(\TT)}
	\norm{D_tF}_{L^2(\TT)}
	\leq C(M).
\end{equation}
For $ I_3 $ in \eqref{eq: dt E1}, an application of Lemma \ref{lem: commutator Dt} gives
\begin{equation}\label{ineq: I3}
	\ba
	I_3
	= & \iT [\fa ,\,\Tslam] D_tF
	\cdot\Tslam F  \rd\bx
	+\iT \fa[D_t,\,\Tslam]F
	\cdot\Tslam F  \rd\bx \\
	\leq & \norm{[\fa,\,\Tslam]D_tF}_{L^2(\TT)}
	\norm{\Tslam F}_{L^2(\TT)} \\
	& +\norm{\fa}_{L^{\infty}(\TT)}
	\norm{[D_t,\,\Tslam]F}_{L^2(\TT)}
	\norm{\Tslam F}_{L^2(\TT)} \\
	\leq & C(M).
	\ea
\end{equation}
Similarly, 
\begin{equation}\label{ineq: I4}
	\ba
	I_4
	= & -\iT [\Ups^{\ka-\frac{3}{2}},\,D_t]D_t  \wb{\p}_if\cdot D_tF\rd\bx 
	-\iT D_t[\Ups^{\ka-\frac{3}{2}},\,D_t]  \wb{\p}_if
	\cdot D_tF\rd\bx  \\
	\leq & 
	\norm{[\Ups^{\ka-\frac{3}{2}},\,D_t]D_t\bp_if}_{L^2(\TT)}\norm{D_tF}_{L^2(\TT)} \\
	& +\norm{D_t[\Ups^{\ka-\frac{3}{2}},\,D_t]\bp_if}_{L^2(\TT)}
	\norm{D_tF}_{L^2(\TT)} \\
	\leq & C(M).
	\ea
\end{equation}
For $ I_5 $ in \eqref{eq: dt E1}, there holds that
\begin{equation}\label{ineq: I5}
	\ba
	I_5
	= & -\iT [\Ups^{\ka-\frac{3}{2}},\,\fa ]  T_{\lam}\wb{\p}_if\cdot D_tF\rd\bx
	-\iT \fa[\Ups^{\ka-\frac{3}{2}},\,T_{\lam}]  \wb{\p}_if\cdot D_tF\rd\bx \\
	\leq & \norm{[\Ups^{\ka-\frac{3}{2}},\,\fa]\Tlam\bp_if}_{L^2(\TT)}
	\norm{D_tF}_{L^2(\TT)} \\
	& +\norm{\fa[\Ups^{\ka-\frac{3}{2}},\,\Tlam]\bp_if}_{L^2(\TT)}
	\norm{D_tF}_{L^2(\TT)} \\
	\leq & C(M).
	\ea
\end{equation}
Next we estimate $ I_6 $ with $ \cN\ppm $ given by \eqref{def: cN}. Since $ \bp_ip=\p_ip+\Hf(\bp_if)\p_3 p $, it can be derived from \eqref{def: harmonic extension} and \eqref{eq: bpi p =0} that
\begin{equation}\label{eq: bpi p elliptic}
\begin{cases}
	\Lap\bp_ip\ppm=\p_i\Lap p\ppm+\Hf\ppm(\bp_if)\p_3\Lap p\ppm
	+2\nabla\Hf\ppm(\bp_if)\cdot\nabla\p_ip\ppm, 
	& \text{in $ \Om\ppm $}, \\
	\bp_ip\ppm=0,
	& \text{on $ \Ga_f $}, \\
	\p_3\bp_ip\ppm=0, 
	& \text{on $ \Ga\ppm $}.
\end{cases}
\end{equation}
Then we can use the elliptic system \eqref{eq: elliptic p}, the estimates \eqref{estim: f} and \eqref{estim: p} to get that
\begin{equation}\label{ineq: Lap bpi p}
	\norm{\Lap\bp_i p\ppm}_{H^{\ka-2}(\Om\ppm)}
	\leq C(M).
\end{equation}
Thus, \eqref{ineq: Lap bpi p} can be applied to \eqref{eq: bpi p elliptic} to yield that
\begin{equation}\label{ineq: bpi p}
	\norm{\nabla_N\bp_ip\ppm}_{H^{\ka-\frac{3}{2}}(\Ga_f)}
	\leq C(M)\norm{\Lap\bp_i p\ppm}_{H^{\ka-2}(\Om\ppm)}
	\leq C(M).
\end{equation}
Therefore, we have
\begin{equation*}
	\norm{\cN\ppm}_{H^{\ka-\frac{3}{2}}(\TT)}
	\leq C(M)
	\norm{(\nabla_N\bp_ip\ppm,\,\bp\rho\ppm,\,\bp u\ppm,\,D_t\bp_if)}_{H^{\ka-\frac{3}{2}}(\TT)}
	\leq C(M).
\end{equation*}
Thus, 
\begin{equation}\label{ineq: I6}
	I_6\leq C(M). 
\end{equation}

Combining all the estimates \eqref{ineq: I1}--\eqref{ineq: I6}, we have
\begin{equation}\label{ineq: energy: f H ka}
	\frac{\rd}{\rd t}\norm{(D_t\bp_if,\,\fa^{1/2}\Tslam\bp_if)}_{H^{\ka-\frac{3}{2}}(\TT)}
	\leq  C(M).
\end{equation}
As for $ \norm{f}_{L^2(\TT)} $, we use \eqref{eq: f: Dt f} to get that
\begin{equation}\label{ineq: energy: f L2}
	\ba
	\frac{1}{2}\frac{\rd}{\rd t}\norm{f}_{L^2(\TT)}^2
	=\iT D_t f\cdot f\rd\bx
	+\iT\frac{\bp_ju_j}{2}|f|^2\rd\bx
	\leq C(M).
	\ea
\end{equation}
Assuming the Taylor sign condition, \eqref{ineq: energy: f H ka}--\eqref{ineq: energy: f L2} and \eqref{eq: fa to taylor sign} prove \eqref{estim: f}.

\section{Full estimates of $ (p,\,u) $}\label{sec: full p u}
In this section, we shall recover the full estimates of $ (p,\,u) $ from the tangential energy estimates in Section \ref{sec: basic} by some elliptic estimates. 

\subsection{Full estimates of $ p $}

To recover the full estimates of the pressure $ p $ from the tangential estimates of $ D_t^lp $ in Section \ref{sec: basic}, we shall rewrite the wave equation of $ p $ in \eqref{eq: wave of p} as
\begin{equation}\label{eq: elliptic p}
	\begin{cases}
		\Lap p\ppm=\frac{\rho\ppm}{\ga p\ppm}D_t^2p\ppm
		-\rho\ppm\tr(\nabla u\ppm)^2
		+\cM_p\ppm,
		& \text{in $ \Om\ppm $}, \\
		p\ppm=q(t),
		& \text{on $ \Ga_f $}, \\
		\p_3p\ppm=0, 
		& \text{on $ \Ga\ppm $},
	\end{cases}
\end{equation}
where $ q(t) $ is given by \eqref{def: q} and
\begin{equation}\label{def: cM p}
	\cM_p\ppm=-\frac{\rho\ppm}{\ga (p\ppm)^2}(D_tp\ppm)^2
	+\frac{1}{\rho\ppm}\nabla\rho\ppm\cdot\nabla p\ppm.
\end{equation}

\begin{prop}\label{prop: full p}
	For $ t\in[0,\,T] $, there holds that
	\begin{equation}\label{estim: p}
		\norm{p}_{H^{\ka}(\Om)}
		+\sum_{l=1}^{\ka+1}
		\norm{D_t^lp}_{H^{\ka+1-l}(\Om)}
		\leq C(M_0)+TC(M).
	\end{equation}
	Furthermore, 
	\begin{equation}\label{estim: Dtu}
		\norm{D_tu}_{H^{\ka-1}(\Om)}
		+\sum_{l=2}^{\ka+1}\norm{D_t^lu}_{H^{\ka+1-l}(\Om)}
		\leq C(M_0)+TC(M).
	\end{equation}
\end{prop}

\begin{proof}

To prove \eqref{estim: p}, we shall use an induction over the index $ l $. 
When $ l=\ka+1 $, it follows from \eqref{estim: Dtl(p, u)} that
\begin{equation}\label{ineq: full p induction: ka+1}
	\norm{(D_t^{\ka+1}p,\,D_t^{\ka+1}u)}_{L^2(\Om)}
	\leq C(M_0)+TC(M). 
\end{equation}
For $ D_t^{\ka}p $, since
\begin{equation*}
	\nabla D_t^{\ka}p
	=\rho D_t^{\ka}(\frac{\nabla p}{\rho})
	-\rho[D_t^{\ka},\,\frac{1}{\rho}]\nabla p
	-[D_t^{\ka},\,\nabla]p,
\end{equation*}
we have from the momentum equation $ D_tu=-\frac{\nabla p}{\rho} $ that
\begin{equation*}
\ba
	\norm{D_t^{\ka}p}_{H^1(\Om)}
	\leq & \norm{\nabla D_t^{\ka}p}_{L^2(\Om)}
	+\norm{D_t^{\ka}p}_{L^2(\Om)} \\
	\leq & \norm{\rho D_t^{\ka+1}u}_{L^2(\Om)}
	+\norm{\rho[D_t^{\ka},\,\frac{1}{\rho}]\nabla p}_{L^2(\Om)} \\
	& +\norm{[D_t^{\ka},\,\nabla]p}_{L^2(\Om)}
	+\norm{D_t^{\ka}p}_{L^2(\Om)} \\
	\leq & \norm{D_t^{\ka+1}u}_{L^2(\Om)}+C(\cF).
\ea
\end{equation*}
Thus, by \eqref{estim: cF} and \eqref{ineq: full p induction: ka+1}, we have
\begin{equation}\label{ineq: full p induction: ka}
	\norm{D_t^{\ka}p}_{H^1(\Om)}
	\leq C(M_0)+TC(M). 
\end{equation}

Assume that $ 1\leq l\leq\ka-1 $ and 
\begin{equation}\label{ineq: full p induction assump}
	\sum_{k=l+1}^{\ka+1}\norm{D_t^kp}_{H^{\ka+1-k}(\Om)}
	\leq C(M_0)+TC(M).
\end{equation}
Then we shall prove that
\begin{equation}\label{ineq: full p induction}
	\norm{D_t^lp}_{H^{\ka+1-l}(\Om)}
	\leq C(M_0)+TC(M).
\end{equation}

When $ 1\leq l\leq\ka-1 $, the equation for $ D_t^lp $ is
\begin{equation}\label{eq: elliptic Dtl p}
	\begin{cases}
		\Lap D_t^lp=\frac{\rho}{\ga p}D_t^{l+2}p
		+[\Lap,\,D_t^l]p
		-[\frac{\rho}{\ga p},\,D_t^l]D_t^2p \\
		\quadd\quadd
		-D_t^l\BL\rho\tr(\nabla u)^2\BR
		+D_t^l\cM_p,
		& \text{in $ \Om\ppm $}, \\
		D_t^lp=\p_t^lq(t),
		& \text{on $ \Ga_f $}, \\
		\p_3 D_t^lp=[\p_3, D_t^l]p, 
		& \text{on $ \Ga\ppm $},
	\end{cases}
\end{equation}
Since
\begin{equation*}
\ba
	[\Lap,\,D_t^l]p
	= & \sum_{m=0}^{l-1}D_t^m[\Lap,\,D_t]D_t^{l-1-m}p \\
	= & \sum_{m=0}^{l-1}D_t^m\BL\Lap u_J\p_J D_t^{l-1-m}p
	+2\nabla u_J \p_J D_t^{l-1-m}p\BR \\
	= & \sum_{m=0}^{l-1}\sum_{n=0}^{m}
	\BL D_t^n\Lap u_J\cdot D_t^{m-n}\p_J D_t^{l-1-m}p
	+2D_t^n\nabla u_J\cdot D_t^{m-n}\p_J D_t^{l-1-m}p\BR,
\ea
\end{equation*}
\begin{equation*}
\ba
	-[\frac{\rho}{\ga p},\,D_t^l]D_t^2p
	= & \sum_{m=0}^{l-1}D_t^m[D_t,\,\frac{\rho}{\ga p}]D_t^{l+1-m}p \\
	= & \sum_{m=0}^{l-1}D_t^m\BL D_t(\frac{\rho}{\ga p})D_t^{l+1-m}p\BR \\
	= & \sum_{n=0}^{l-1}D_t^{n+1}(\frac{\rho}{\ga p})
	\cdot D_t^{l+1-n}p,
\ea
\end{equation*}
\begin{equation*}
\ba
	D_t^l(\rho\tr(\nabla u)^2)
	= & D_t^l(\rho \p_Ju_K \p_K u_J) \\
	= & \sum_{m=0}^{l}\sum_{n=0}^{l-m}
	D_t^{l-m-n}\rho\cdot D_t^m\p_Ju_K\cdot D_t^n\p_K u_J,
\ea
\end{equation*}
we have
\begin{equation}\label{estim: full p: Lap Dtlp}
	\norm{\Lap D_t^lp}_{H^{\ka-1-l}(\Om)}
	\leq C(\cF)\norm{D_t^{l+2}p}_{H^{\ka-1-l}(\Om)}
	+C(\cF).
\end{equation}
Similarly, 
\begin{equation*}
\ba
	[\p_3,\,D_t^{l}]p
	= & \sum_{m=0}^{l-1}D_t^m[\p_3,\,D_t]D_t^{l-1-m}p \\
	= & \sum_{m=0}^{l-1}D_t^m\BL\p_3u_J \p_JD_t^{l-1-m}p\BR \\
	= & \sum_{m=0}^{l-1}\sum_{n=0}^{m}
	D_t^n\p_3u_J\cdot D_t^{m-n}\p_JD_t^{l-1-m}p,
\ea
\end{equation*}
yields that
\begin{equation}\label{estim: full p: p3 Dtlp}
	\norm{\p_3 D_t^lp}_{H^{\ka-\frac{1}{2}-l}(\Ga\ppm)}
	\leq C(\cF).
\end{equation}
On the interface $ \Ga_f $, the fact that $ p|_{\Ga_f}=q(t) $ which is independent on $ \bx $ infers that
\begin{equation}\label{estim: full p: Dtlp Gaf}
\ba
	\norm{D_t^lp}_{H^{\ka+\frac{1}{2}-l}(\Ga_f)}
	= & \norm{D_t^lp}_{L^2(\Ga_f)} \\
	\leq & C(\norm{f}_{H^{\ka-\frac{1}{2}}})
	\norm{D_t^lp}_{H^1(\Om)}
	\leq C(\norm{f}_{H^{\ka-\frac{1}{2}}})C(\cF).
\ea
\end{equation}
Therefore, the standard elliptic theory can be applied to \eqref{eq: elliptic Dtl p} to get that
\begin{equation}\label{estim: Dtl p}
	\ba
		\norm{D_t^lp}_{H^{\ka+1-l}(\Om)}
		\leq & C(\norm{f}_{H^{\ka-\frac{1}{2}}})
		\BL\norm{\Lap D_t^l p}_{H^{\ka-1-l}(\Om)} \\
		&+\norm{\p_3D_t^lp}_{H^{\ka-\frac{1}{2}-l}(\Ga\ppm)}
		+\norm{D_t^lp}_{H^{\ka+\frac{1}{2}-l}(\Ga_f)}\BR \\
		\leq  & C(M_0)+TC(M).
	\ea
\end{equation}
where we have used \eqref{estim: full p: Lap Dtlp}--\eqref{estim: full p: Dtlp Gaf}, the induction assumption \eqref{ineq: full p induction assump}, \eqref{estim: cF} and \eqref{estim: f}.

The case of $ l=0 $ follows in a similar way. Notice that we can only get $ \norm{p}_{H^{\ka}(\Om)} $ instead of $ \norm{p}_{H^{\ka+1}(\Om)} $ due to limited regularity of the interface $ f $.  Thus, \eqref{estim: p} is proved.

To prove \eqref{estim: Dtu}, since $ D_tu=-\frac{\nabla p}{\rho} $ and 
\begin{equation*}
	D_t^{l+1}u=-D_t^l(\frac{1}{\rho}\nabla p)
	=-\frac{1}{\rho}\nabla D_t^l p
	-[D_t^l,\,\frac{1}{\rho}\nabla]p,
\end{equation*}
the estimates of $ u $ in \eqref{estim: Dtu} follow from \eqref{estim: p} and \eqref{estim: cF}. 


\end{proof}

\subsection{Full estimates of $ u $}

The full estimates of $ u $ can be recovered by Lemma \ref{lem: elliptic u} and the fact that
\begin{equation}\label{eq: Dtbpi f}
	\bp_iu\cdot N=\bp_i D_tf-\bp_jf\bp_iu_j
	=D_t\bp_if,
	\quadd i=1,\,2.
\end{equation} 

\begin{prop}\label{prop: full u}
	For $ t\in[0,\,T] $, there holds that
	\begin{equation}\label{estim: u Hk}
		\norm{u}_{H^{\ka}(\Om)}
		\leq C(M_0)+TC(M).
	\end{equation}
\end{prop}

\begin{proof}
	We apply \eqref{ineq: elliptic: u} and \eqref{eq: Dtbpi f} to get that
	\begin{equation*}
	\ba
		\norm{u}_{H^{\ka}(\Om)}
		\leq & C(\norm{f}_{H^{\ka-\frac{1}{2}}})
		\BL\norm{\cur u}_{H^{\ka-1}(\Om)}
		+\norm{\dvg u}_{H^{\ka-1}(\Om)}
		\\ & \quadd\quadd 
		+\norm{\bp_iu\cdot N}_{H^{\ka-\frac{3}{2}}(\TT)}
		+\norm{u}_{L^2(\Om)}\BR \\
		\leq & C(\norm{f}_{H^{\ka-\frac{1}{2}}})
		\BL\norm{\om}_{H^{\ka-1}(\Om)}
		+\norm{\frac{1}{\ga p}D_tp}_{H^{\ka-1}(\Om)}
		\\ & \quadd\quadd 
		+\norm{D_t\bp_if}_{H^{\ka-\frac{3}{2}}(\TT)}
		+\norm{u}_{L^2(\Om)}\BR \\
		\leq & 
		 C\Big(C(M_0)+TC(M)\Big)\cdot \BL C(M_0)+TC(M)\BR \\
		\leq & C(M_0)+TC(M),
	\ea
	\end{equation*}
	where we have used \eqref{estim: cF}, \eqref{estim: Dtl(p, u)}, \eqref{estim: om} and \eqref{estim: f}. 
\end{proof}

\appendix

\section{Elliptic estimates}\label{sec: elliptic}

For the one-phase elliptic system
\begin{equation}\label{eq: div-curl}
	\begin{cases}
		\cur u=\om, \quad \dvg u=\sigma,
		& \text{in $ \Om_f^- $}, \\
		u\cdot N=\theta,
		& \text{on $ \Ga_f $}, \\
		u\cdot n_-=0, \quad
		\iTT u_j\rd\bx=\al_j \,(j=1,\,2),
		& \text{on $ \Ga^- $}, \\
	\end{cases}
\end{equation}
we have the following existence result given by Proposition 5.1 in \cite{MR3745155} (see also \cite{MR3685967,MR2388661}):
\begin{lem}\label{lem: div-curl}
	Assume that $ f\in H^{\ka-\frac{1}{2}}(\TT) $ with $ \ka>\frac{5}{2} $. For $ s\in[2,\,\ka] $, let $ (\om,\,\sigma)\in H^{s-2}(\Om_f^-) $ and $ \theta\in H^{s-\frac{3}{2}}(\TT) $ be such that
	\begin{equation*}
		\int_{\Om_f^-}\sigma\rd x
		=\int_{\TT}\theta\rd\bx,
	\end{equation*}
	\begin{equation*}
		\dvg\om=0, \,\text{in $ \Om_f^- $},
		\quad
		\int_{\Ga^-}\om_3\rd\bx=0.
	\end{equation*}
	Then there exists a unique $ u\in H^{s-1}(\Om_f^-) $ to the system \eqref{eq: div-curl} such that
	\begin{equation}\label{ineq: elliptic u -1}
		\norm{u}_{H^{s-1}(\Om_f^-)}
		\leq C(\norm{f}_{H^{\ka-\frac{1}{2}}})
		\BL\norm{(\om,\,\sigma)}_{H^{s-2}(\Om_f^-)}
		+\norm{\theta}_{H^{s-\frac{3}{2}}(\Ga_f)}
		+|\al_1|+|\al_2|\BR.
	\end{equation}
\end{lem}

The regularity of the solution of the one-phase elliptic system \eqref{eq: div-curl} was improved in \cite{MR3685967} (see also \cite{Wang-Zhang-Zhao-Mach}) by using tangential derivatives for the boundary condition on the surface $ \Ga_f $:
\begin{lem}\label{lem: elliptic u}
	Assume that $ f\in H^{\ka-\frac{1}{2}}(\TT) $ with $ \ka>\frac{5}{2} $.  
	For $ s\in[2,\,\ka] $, there holds that
	\begin{equation}\label{ineq: elliptic: u}
		\ba
		\norm{u}_{H^{s}(\Om_f^-)}
		\leq & C(\norm{f}_{H^{\ka-\frac{1}{2}}})\BL
		\norm{\cur u}_{H^{s-1}(\Om_f^-)}
		+\norm{\dvg u}_{H^{s-1}(\Om_f^-)} \\
		& \quadd
		+\sum_{i=1,2}\norm{\bp_i u\cdot N }_{H^{s-\frac{3}{2}}(\Ga_f)}
		+\norm{u}_{L^2(\Om_f^-)}\BR.
		\ea
	\end{equation}
\end{lem}
Clearly, these two results also hold for the one-phase elliptic systems in $ \Om^+_f $ in a similar fashion.

\section{Paradifferential operators and commutator estimates}\label{sec: paradiff}

%

In this appendix, we shall recall some basic facts on paradifferential operators from \cite{MR2418072}.

We first introduce the symbols with limited spatial smoothness. Let $ W^{k,\infty}(\R^d) $ be the usual Sobolev spaces for $ k\in\mathbb{N} $.
\begin{defn}
	Given $ \mu\in[0,\,1] $ and $ m\in\R $, we denote by $ \Ga^m_{\mu}(\R^d) $ the space of locally bounded functions $ a(x,\,\xi) $ on $ \R^d\times\R^d\backslash\{0\} $, which are $ C^{\infty} $ with respect to $ \xi $ for $ \xi\neq0 $ such that, for all $ \al\in\mathbb{N}^d $ and $ \xi\neq0 $, the function $ x\to\p_{\xi}^{\al}a(x,\,\xi) $ belongs to $ W^{\mu,\infty} $ and there exists a constant $ C_{\al} $ such that
	\begin{equation*}
		\norm{\p_{\xi}^{\al}a(\cdot,\,\xi)}_{W^{\mu,\infty}}
		\leq C_{\al}(1+|\xi|)^{m-|\al|},
		\quadd
		\forall\,|\xi|\geq\frac{1}{2}.
	\end{equation*}
	The seminorm of the symbol is defined as
	\begin{equation*}
		M_{\mu}^m(a)
		:=\sup_{|\al|\leq\frac{3d}{2}+1+\mu}
		\sup_{|\xi|\geq\frac{1}{2}}
		\norm{(1+|\xi|)^{-m+|\al|}\p_{\xi}^{\al}a(\cdot,\,\xi)}_{W^{\mu,\infty}}
	\end{equation*}
	If $ a $ is a function independent of $ \xi $, then
	\begin{equation*}
		M^0_{\mu}(a)=\norm{a}_{W^{\mu,\infty}}.
	\end{equation*}
\end{defn}

\begin{defn}
	Given a symbol $ a $, the paradifferential operator $ T_a $ is defined by
	\begin{equation}\label{def: Ta operator}
		\wh{T_au}(\xi)
		:=(2\pi)^{-d}\int_{\R^d}
		\chi(\xi-\eta,\,\eta)
		\wh{a}(\xi-\eta,\,\eta)
		\psi(\eta)\wh{u}(\eta)\rd\eta,
	\end{equation}
	where $ \wh{a} $ is the Fourier transform of $ a $ with respect to the first variable. $ \chi(\xi,\,\eta)\in C^{\infty}(\R^d\times\R^d) $ is an admissible cutoff function, that is, there exist $ 0<\veps_1<\veps_2 $ such that
	\begin{equation*}
		\chi(\xi,\,\eta)=1
		\quad
		\text{for $ |\xi|\leq\veps_1|\eta| $},
		\quadd
		\chi(\xi,\,\eta)=0
		\quad
		\text{for $ |\xi|\geq\veps_2|\eta| $},
	\end{equation*}
	and
	\begin{equation*}
		|\p_{\xi}^{\al}\p_{\eta}^{\beta}\chi(\xi,\,\eta)|
		\leq C_{\al,\beta}
		(1+|\eta|)^{-|\al|-|\beta|}
		\quad
		\text{for $ (\xi,\,\eta)\in\R^d\times\R^d $}.
	\end{equation*}
	The cutoff function $ \psi(\eta)\in C^{\infty}(\R^d) $ satisfies
	\begin{equation*}
		\psi(\eta)=0
		\quad
		\text{for $ |\eta|\leq 1 $},
		\quadd
		\psi(\eta)=1
		\quad
		\text{for $ |\eta|\geq 2 $}.
	\end{equation*}
\end{defn}

The admissible cutoff function $ \chi(\xi,\,\eta) $ can be chosen as
\begin{equation*}
	\chi(\xi,\,\eta)=\sum_{k=0}^{\infty}
	\zeta_{k-3}(\xi)\vphi(\eta),
\end{equation*}
where $ \zeta(\xi)=1 $ for $ |\xi|\leq1.1 $, $ \zeta(\xi)=0 $ for $ |\xi|\geq 1.9 $, and
\begin{equation*}
\begin{cases}
	\zeta_k(\xi)=\zeta(2^{-k}\xi)
	& \text{for $ k\in\mathbb{Z} $}, \\
	\vphi_0=\zeta,
	\quad
	\vphi_k=\zeta_k-\zeta_{k-1}
	& \text{for $ k\geq 1 $}.
\end{cases}
\end{equation*}

We also introduce the Littlewood-Paley operators $ \Delta_k,\,S_k $ defined as
\begin{equation*}
	\Delta_k u =\mathcal{F}^{-1}(\vphi_k\wh{u})
	\quad
	\text{for $ k\geq0 $},
	\quadd
	\Delta_k u=0
	\quad
	\text{for $ k<0 $},
\end{equation*}
\begin{equation*}
	S_ku=\sum_{l\leq k}\Delta_lu
	\quad
	\text{for $ k\in\mathbb{Z} $}.
\end{equation*}

When the symbol $ a $ depends only on the first variable $ x $ in $ T_au $, we take $ \psi=1 $ in \eqref{def: Ta operator}. Then $ T_au $ is just usual Bony's paraproduct defined as
\begin{equation}\label{def: Tau paraproduct}
	T_au=\sum_{k=0}S_{k-3}a \Delta_ku.
\end{equation}
We have the following Bony's paraproduct decomposition:
\begin{equation}\label{def: au Bony's paraproduct}
	au=T_au+T_ua+R(a,\,u),
\end{equation}
where the remainder term $ R(a,\,u) $ is
\begin{equation*}
	R(a,\,u)=\sum_{|k-l|\leq2}\Delta_ka\Delta_lu.
\end{equation*}

\begin{lem}\label{lem: Bony product}
	There holds that
	\begin{enumerate}
	\item If $ s\in\R $ and $ \sigma<\frac{d}{2} $, then
	\begin{equation*}
		\norm{T_au}_{H^s}
		\les\min\{
		\norm{a}_{L^{\infty}}\norm{u}_{H^s},
		\,
		\norm{a}_{H^{\sigma}}\norm{u}_{H^{s+\frac{d}{2}-\sigma}},
		\,
		\norm{a}_{H^{\frac{d}{2}}}\norm{u}_{H^{s+}}\}.
	\end{equation*}

	\item If $ s>0 $ and $ s_1,\,s_2\in\R $ with $ s_1+s_2=s+\frac{d}{2} $, then
	\begin{equation*}
		\norm{R(a,\,u)}_{H^s}
		\les\norm{a}_{H^{s_1}}\norm{u}_{H^{s_2}}.
	\end{equation*}

	\item  If $ s>0 $, $ s_1\geq s $, $ s_2\geq s $ and $ s_1+s_2=s+\frac{d}{2} $, then
	\begin{equation}\label{ineq: Sobolev au}
		\norm{au}_{H^{s}}
		\les\norm{a}_{H^{s_1}}\norm{u}_{H^{s_2}}.
	\end{equation}
	
	\end{enumerate}
	
\end{lem}

There is also the symbolic calculus of paradifferential operator in Sobolev spaces.
\begin{lem}\label{lem: symbol calculus}
	Let $ m,\,m'\in\R $.
	\begin{enumerate}
	\item If $ a\in\Ga^m_0(\R^d) $, then for any $ s\in\R $,
	\begin{equation*}
		\norm{T_a}_{H^s\to H^{s-m}}
		\les M^m_0(a).
	\end{equation*}
	
	\item if $ a\in\Ga^m_{\rho}(\R^d) $ and $ b\in\Ga^{m'}_{\rho}(\R^d) $ for $ \rho>0 $, then for any $ s\in\R $,
	\begin{equation*}
		\norm{T_aT_b-T_{a\sharp b}}_{H^s\to H^{s-m-m'+\rho}}
		\les M^m_{\rho}(a)M^{m'}_{0}(b)
		+M^m_{0}(a)M^{m'}_{\rho}(b),
	\end{equation*}
	where
	\begin{equation*}
		a\sharp b=\sum_{|\al|<\rho}
		\p_{\xi}^{\al}a(x,\,\xi)
		D_x^{\al}b(x,\,\xi),
		\quadd
		D_x=-\ri\p_x.
	\end{equation*}
	
	\item If $ a\in\Ga^m_{\rho}(\R^d) $ for $ \rho\in(0,\,1] $, then for any $ s\in\R $,
	\begin{equation*}
		\norm{(T_a)^*-T_{a^*}}_{H^s\to H^{s-m+\rho}}
		\les M^m_{\rho}(a),
	\end{equation*}
	where $ (T_a)^* $ is the adjoint operator of $ T_a $ and $ a^* $ is the complex conjugate of the symbol $ a $.
	\end{enumerate}
\end{lem}

To estimate commutators, we recall a lemma from \cite{MR3260858} (Lemma 2.15).
\begin{lem}\label{lem: commutator Dt}
	Consider a symbol $ p=p(t,\,x,\,\xi) $ which is homogeneous of order $ m $. There holds that
	\begin{equation}\label{ineq: commutator with Dt}
		\norm{[T_p,\,\p_t+T_u\cdot\nabla]u}_{H^m}
		\les \BL M_0^m(p)\norm{u}_{C^{1+}_{\star}}
		+M_0^m(D_t p)\BR
		\norm{u}_{H^m}.
	\end{equation}
\end{lem}

\section*{Declaration}

\textbf{Conflict of interest}: On behalf of all authors, the corresponding author states that there is no conflict of interest. Our manuscript has no associated data.

\bibliographystyle{abbrv}

\end{document}